\documentclass[lettersize,onecolumn]{IEEEtran}
\usepackage{amsmath,amsfonts}
\usepackage{algorithmic}
\usepackage{algorithm}
\usepackage{array}
\usepackage{amsthm}
\usepackage[caption=false,font=normalsize,labelfont=sf,textfont=sf]{subfig}
\usepackage{textcomp}
\usepackage{stfloats}
\usepackage{url}
\usepackage{verbatim}
\usepackage{graphicx}
\usepackage{cite}
\newtheorem{theorem}{Theorem}
\newtheorem{corollary}{Corollary}
\newtheorem{lemma}{Lemma}
\newtheorem{definition}{Definition}
\newtheorem{proposition}{Proposition}
\usepackage{amssymb}
\begin{document}

\title{A construction method for 2-phase and 4-phase Golay Complementary Sequences}

\author{Xu Wang
\thanks{The author is with School of Electronics \& Information Technology, Nanjing University of Information Science and Technology, Nanjing 210044,China(email:003796@nuist.edu.cn).}
\thanks{Manuscript received , 2026; revised , 2026.}}

\markboth{Journal of \LaTeX\ Class Files,~Vol.~1, No.~2, December~2023}%
{Shell \MakeLowercase{\textit{et al.}}: A Sample Article Using IEEEtran.cls for IEEE Journals}

\IEEEpubid{0000--0000~\copyright~2023 IEEE}

\maketitle

\begin{abstract}
Motivated by the recent constructions of Golay Complementary Array (GCA), we extend the "Three-stage construction" proposed by Fiedler and Jedwab to a more general form. All the known 2-phase Golay Complementary Sequences (GCSs) of length $2^a10^b26^c,a,b,c\ge 0$ and 4-phase GCSs of length $2^{a+u}3^b5^c11^d13^e$ where $a,b,c,d,e,u\ge0,b+c+d+e\le a+2u+1,u\le c+e$ can be constructed in a uniform method from five 2-phase seeds of length 2,10,10,20,26 and five 4-phase seeds of length 3,5,8,11,13. Furthermore, additional 2-phase and 4-phase GCSs of many lengths are produced and a closed form for numbers of 2-phase and 4-phase GCSs is given. 
\end{abstract}


\section{Introduction}
\IEEEPARstart{G}{olay} complementary sequence(GCS) was originally introduced by Golay \cite{golay1951} in 1951 to solve the optical problem of infrared multislit spectrometry. Owing to their autocorrelation property, GCSs have found wide applications, including reduction of the peak-to-average power ratio(PAPR) in Orthogonal Frequency Division Multiplexing(OFDM) systems\cite{Davis}, omnidirectional precoding for antenna arrays\cite{Du}, surface wave suppression, and the construction of Hadamard matrix\cite{Turyn}.

The study of Golay focuses on the GCS with elements of {+1,-1}. Such sequences are referred to as 2-phase GCSs. Golay\cite{golay1961} proved that 2-phase GCS must have even length and demonstrated the existence of two pairs of length 10 and one pair of length 26 GCS. He also gave a construction of lengths 2n and 2mn from existing pairs of length m and n. In 1974, Turyn\cite{Turyn} proposed a construction for length mn. With Turyn's work, it is known that 2-phase GCSs exist for lengths $2^a10^b26^c$, where $a,b,c\ge 0$, known as 2-phase Golay numbers. In 1990, Eliahou\cite{Eliahou} proved that 2-phase GCSs do not exist for lengths divisible by a prime congruent to 3 $ \pmod{4}$. In 2003, Borwein and Ferguson\cite{Borwein} presented a complete computer search for all lengths less than 100, and the results showed that for $n\le100$, the only possible lengths for 2-phase GCSs are $2^a10^b26^c$, where $a,b,c\ge 0$. It is also conjectured that the 2-phase Golay numbers are the only possible lengths for 2-phase GCSs. 

To cover more lengths, researchers began to study GCSs with elements of $\{1,i,-1,-i\}$, which are known as 4-phase GCSs, complex GCSs, quaternary GCSs, or multiphase GCSs in some literature. Early work on 4-phase GCSs can be traced back to 1978, when Sivaswamy \cite{Sivaswamy} gave an example of a length-3 4-phase GCS and a construction method for length $n2^k$, where n is the length of a 4-phase GCS. In 1980, Frank\cite{Frank} further gave examples of length-5 and length-13 4-phase GCSs, and also proposed construction for length 2mn, where m and n are lengths of a 4-phase GCS. In 2002, Craigen\cite{Craigen} reported a construction that combines  two 4-phase GCSs of lengths $g_1$ and $g_2$ with a 2-phase GCS of length g to obtain a GCS of length $gg_1g_2$. They also performed a computer search for lengths $n\le 19$ and 21. As a result, 4-phase GCSs are known to exist for lengths $2^{a+u}3^b5^c11^d13^e$, where $a,b,c,d,e,u\ge0,b+c+d+e\le a+2u+1,u\le c+e$.

The construction results above establish the existence of 4-phase GCSs for certain lengths, but they do not cover all existing 4-phase GCSs. In 1999, David and Jedwab\cite{Davis} discovered an important connection between GCS and Reed-muller codes, and gave an explicit algebraic normal form for $m!2^{h(m+2)}$ ordered GCSs of length $2^m$ over $\mathbb{Z}_{2^h}$. These sequences are called standard GCSs in later works. In 2003, Borwein and Ferguson\cite{Borwein} gave a construction of $2^{k+6}(k+1)(k+1)!$ 2-phase GCSs at length $2^kn$, where n is a primitive Golay seed pair of length 10,10,20 and 26. In 2005, Li and Chu\cite{Li} unexpectedly found 1024 nonstandard 4-phase GCSs of length 16 by computer search. This phenomenon was explained by Fiedler and Jedwab\cite{Fiedler2006} through the "crossover" of one class of length-8 4-phase GCSs. In 2008, Fiedler,Jedwab and Parker\cite{Fiedler20081}\cite{Fiedler2008} proposed the important viewpoint that GCSs are naturally projections of multi-dimensional GCAs, and introduced the "3-stage construction" framework for constructing GCSs. Using this framework, they determined the number of 4-phase GCS of length $2^m$. In 2011, Gibson and Jedwab\cite{Gibson} applied the "3-stage construction" with 4-phase GCS seeds of lengths 3,5,11,13. They obtained the number of 4-phase GCSs of length $s^c2^m$, where $m\ge1,1\le c\le m+1,s\in \{3,5,11,13\}$. They also used a "sum-difference" construction to explain the 4-phase GCS of length 10,20 and 26. 

Recently, Du and Jiang\cite{Du} proposed a more general construction of 4-phase GCAs. While their work focuses on the possible sizes of 4-phase GCAs, we focus on constructing as many GCSs as possible for all admissible lengths. Our approach is based on the "3-stage construction" framework used in many previous works. The construction proceeds in three steps: 1. construct GCAs from GCSs 2.apply equivalence operations to obtain more GCAs 3.project the GCAs to GCSs. Our method has two improvements: 1. A more general GCA construction is proposed, 2.Additional equivalence operation is considered during the construction. These improvements bridge the gap between known construction methods and the known admissible lengths for 2-phase and 4-phase GCSs. We also derive closed forms for the numbers of 2-phase and 4-phase GCSs, which answers the question left in the last part of \cite{Borwein} and part of the first question in \cite{Fiedler2008}.

The paper is organized as follows:
In Section II we introduce the notation and definitions used in the paper. In Section III, we review the Three-Stage construction framework, and extend the method to a more general form, the equivalence operations is also discussed. In Section IV, we give a general construction for the 2-phase GCSs of length $2^a10^b26^c$ and 4-phase GCSs of length $2^{a+u}3^b5^c11^d13^e,$. The construction not only explains all the known 2-phase and 4-phase GCSs, but also gives a lower bound of the number for GCSs on the given lengths. In Section V, a conclusion is given.

\section{Notations and Definitions}
Our definitions follow \cite{Fiedler2008} and \cite{Du}. Let $\xi$ denote $e^{2\pi i/H}$, where $H$ represents an even integer throughout this paper. We call the finite set $S=\{1,\xi,\xi^2,...,\xi^{H-1}\}$ the $H$-phase alphabet. A sequence $\mathcal{A}=(a_0,a_1,...a_{n-1})$, where each $a_i\in S$, is called a $H$-phase sequence of length n. For example, the 2-phase alphabet consists of $\{1,-1\}$, and 4-phase alphabet consists of $\{1,i,-1,-i\}$, where $i=\sqrt{-1}$. 

The aperiodic autocorrelation function of a length-s sequence $\mathcal{A}$=(A[j]) is defined as
\begin{equation}   \mathcal{C}_A(u)=\sum_{i}A[i]\overline{A[i+u]}\quad for\quad integers\quad u
\end{equation}
where the bar represents complex conjugation.

\begin{definition}
Let $\mathcal{A}$ and $\mathcal{B}$ be a pair of $H$-phase sequences of length $s$.  If
\begin{equation}
    \mathcal{C}_A(u)+\mathcal{C}_B(u)=0, \quad for\quad all(u)\neq(0,...,0)
\end{equation}
then $\mathcal{A}$ and $\mathcal{B}$ are called a Golay complementary sequence(GCS).
\end{definition}

Let an array of size $s_1\times\cdot\cdot\cdot\times s_r$ be an r-dimensional matrix $\mathcal{A}$=(A[$i_1,...,i_r$]) of complex-valued entries, where $i_1,...,i_r$ are integers. And A[$i_1,...,i_r$]=0, for any $k\in{1,2,..,r}$, $i_k<0$ or $i_k\ge s_k$. A sequence can be viewed as a 1-dimensional array.

The aperiodic autocorrelation function of an $s_1\times\cdot\cdot\cdot\times s_r$ array $\mathcal{A}$=(A[$i_1,...,i_r$]) is defined as 
\begin{equation}    \mathcal{C}_A(u_1,...,u_r)=\sum_{i_1}\cdot\cdot\cdot\sum_{i_r}A[i_1,...,i_r]\overline{A[i_1+u_1,...,i_r+u_r]}\quad for\quad integers\quad u_1,...,u_r
\end{equation}

\begin{definition}
Let a pair of H-phase r-dimensional arrays $\mathcal{A}$ and $\mathcal{B}$ of size $s_1\times\cdot\cdot\cdot\times s_r$, and
\begin{equation}
    \mathcal{C}_A(u_1,...,u_r)+\mathcal{C}_B(u_1,...,u_r)=0, \quad for\quad all(u_1,...,u_r)\neq(0,...,0)
\end{equation}
then we call array A and B Golay complementary array(GCA).
\end{definition}

\section{Construction framework}
In this section, we describe our construction based on the Three-Stage construction framework: 1. Construct GCA from GCSs; 2. Apply Equivalence operations to generate more GCAs; 3. Take projections from GCA to GCSs. This framework has been used in many previous constructions. In this paper, we mainly extend the first two stages by introducing a more general GCA construction and a more general equivalence operation.

\subsection{Construction of GCA from GCSs}
The original method introduced by Fiedler et al. constructs a $(2m+1)$-dimensional Golay array over $\mathbb{Z}_H$ from m+1 Golay sequence pairs over $\mathbb{Z}_H$:

\begin{theorem}\cite[Theorem 7]{Fiedler2008}
\label{Theorem:1}
Let $m\ge1$ be an integer.Suppose that the sequences ($a_k[j_k]$) and ($b_k[j_k]$) form a Golay sequence pair of length $s_k$ over $\mathbb{Z}_H$, for $k=0,1,...,m$. Then the arrays ($f_m[j_0,...,j_m,x_1,...,x_m]$) and ($g_m[j_0,...,j_m,x_1,...,x_m]$) of size $s_0\times\cdot\cdot\cdot\times s_m\times2\times\cdot\cdot\cdot\times 2$ (in which m copies of 2 appear) over $\mathbb{Z}_H$ given by
\begin{equation}
\label{Equation:27}
\begin{aligned}
&f_m[j_0,...,j_m,x_1,...,x_m]=\sum_{k=1}^{m-1}(a_k[j_k]+a_k^*[j_k]-b_k[j_k]-b_k^*[j_k]+\frac{H}{2})x_kx_{k+1}\\&+\sum_{k=1}^{m}(b^*_{k-1}[j_{k-1}]+b_k[j_k]-a_{k-1}[j_{k-1}]-a_k[j_k]+\frac{H}{2})x_k+\sum_{k=0}^{m}a_k[j_k],\\
&g_m[j_0,...,j_m,x_1,...,x_m]=f'_m[j_0,...,j_m,x_1,...,x_m]+\frac{H}{2}x_1,
\end{aligned}
\end{equation}
form a Golay array pair, where $f'_m[j_0,...,j_m,x_1,...,x_m]$ is formed from $f_m[j_0,...,j_m,x_1,...,x_m]$ by interchanging $a_0[j_0],b_0[j_0]$ and by interchanging $a_0^*[j_0],b_0^*[j_0]$.
\end{theorem}

A more general construction of GCA was recently given by Du and Jiang \cite{Du}:
\begin{theorem}\cite[Theorem~11]{Du}
\label{Theorem:2}
Given a nontrivial 2-phase GCA pair {A,B} of size $s_1\times\cdot\cdot\cdot\times s_r$,two H-phase GCA pairs {C,D} and {E,F} of size $t_1\times\cdot\cdot\cdot\times t_r$ and $u_1\times\cdot\cdot\cdot\times u_r$ respectively, suppose
\begin{equation}
\label{eqthm2}
\begin{aligned}
P &= \frac{1}{4}\bigl[ A + B + (B^{*} - A^{*}) \bigr],\
Q = \frac{1}{4}\bigl[ A + B - (B^{*} - A^{*}) \bigr],\\
X &= P \otimes C + Q \otimes D,\
Y = Q^{*} \otimes C - P^{*} \otimes D,\\
G &= X \otimes E + Y \otimes F,\
H = Y^{*} \otimes E - X^{*} \otimes F.
\end{aligned}
\end{equation}
where $*$ denotes reverse conjugation and $\otimes$ denotes Kronecker product.
then {G,H} is an H-phase GCA of size $s_1t_1u_1\times\cdot\cdot\cdot\times s_rt_ru_r$.
\end{theorem}

The proof of Theorem~\ref{Theorem:2} is given in \cite[Theorem~11]{Du}. Since our goal is to construct GCAs from a set of GCSs. We obtain the following corollary:
\begin{corollary}
\label{corollary:1}
Given a nontrivial 2-phase GCS A,B of length $s_{r+1}$, an H-phase r-dimensional GCA C,D of size $s_1\times\cdot\cdot\cdot\times s_r$ and an H-phase GCS E,F of length $s_{r+2}$, we can construct an H-phase $(r+2)$-dimensional GCA of size  $s_1\times\cdot\cdot\cdot\times s_r\times s_{r+1}\times s_{r+2}$.
\end{corollary}
\begin{proof}
Extend C,D to an (r+2)-dimensional GCA of size $s_1\times\cdot\cdot\cdot\times s_r\times1^{(2)}$, extend GCS A,B to (r+2)-dimensional GCA of size $1^{(r)}\times s_{r+1}\times1$, extend GCS E,F to (r+2)-dimensional GCA of size $1^{(r+1)}\times s_{r+2}$. Using these extended (r+2)-dimensional GCAs as entries of Theorem~\ref{Theorem:2}, we obtain the (r+2)-dimensional GCA of size $s_1\times\cdot\cdot\cdot\times s_r\times s_{r+1}\times s_{r+2}$.
\end{proof}

By applying Corollary~\ref{corollary:1} repeatedly, we can construct a $(2r+1)$-dimensional GCA from $r$ nontrivial 2-phase GCSs and $r+1$ H-phase GCSs:
\begin{theorem}
\label{Theorem:3}
Given $r$ nontrivial 2-phase GCSs of length $s_1,...,s_r$ and $r+1$ H-phase GCSs of length $t_1,...,t_{r+1}$, we can construct an H-phase $(2r+1)$-dimensional GCA of size  $t_1\times s_1\times t_2\times\cdot\cdot\cdot\times s_r\times t_{r+1}$.
\end{theorem}
\begin{proof}
Take the H-phase GCS of length t1 and t2 and 2-phase GCS of length s1 as entries of Corollary \ref{corollary:1}. This gives a GCA (G1,H1) of size $t_1\times s_1\times t_2$.

Next, use (G1,H1), the H-phase GCS of length t3, and the 2-phase GCS of length s2 as entries of Corollary \ref{corollary:1}. This yields a GCA (G2,H2) of size $t_1\times s_1\times t_2\times s_2\times t_3$. 

Repeating this process r times and we will get the $(2r+1)$-dimensional GCA of size $t_1\times s_1\times t_2\times\cdot\cdot\cdot\times s_r\times t_{r+1}$.
\end{proof}

Theorem ~\ref{Theorem:1} can be viewed as a special case of Theorem ~\ref{Theorem:3}, where all 2-phase GCSs are of length 2. 

\subsection{Equivalence operations}

An equivalence operation for GCS and GCA is a transformations that preserve the GCS and GCA property. It is well known that the following five equivalence operations can preserve the GCS property.

\begin{proposition}\label{proposition1}\cite[Proposition 8]{Bright}
Let ([$a_1,...,a_{n}$]),([$b_1,...,b_{n}$])be a H-phase GCS, then the following are also H-phase GCS:\\
E1. (Reversal) $([a_n,...,a_1],[b_n,...,b_1])$. \\
E2. (Conjugate Reverse A) $([a_n,...,a_1], [b_1,...,b_n])$. \\
E3. (Swap) $([b_1,...,b_n], [a_1,...,a_n])$.  \\
E4. (Scale A) $([\xi a_1,...,\xi a_n],[b_1,...,b_n])$. \\
E5. (Positional Scaling) $(\xi\star A,\xi\star B)$ where $c\star A$ denotes the sequence of coefficients of the polynomial A(cz), i.e., $[a_1, ca_2,c^2a_3...,c^{n-1}a_n]$.
\end{proposition}

We can extend Proposition~\ref{proposition1} to GCA over $\mathbb{Z}_H$. The following operations preserve zero autocorrelation and can be verified directly from the definition of $\mathcal{C}_A$ and $\mathcal{C}_B$.

\begin{proposition}\label{proposition2}
Let $A[i_1,...,i_r]$ and $B[i_1,...,i_r]$ be Golay complementary array over $\mathbb{Z}_H$, then the following are also Golay complementary arrays:\\
E1. (Reversal on k-th dimension) For $k\in [1,2,...,r]$, $(A[i_1,...,s_k-1-i_k,...,i_r],B[i_1,...,s_k-1-i_k,...,i_r])$.\\
E2. (Conjugate Reverse A)$A^*$,B\\
E3. (Swap)$(B,A)$\\
E4. (Scale A) ($\xi A$,$B$).\\
E5. (Positional Scaling on k-th dimension)For $k\in {1,2,...,r}$,$(A[i_1,...,i_r]\xi^{i_k-1},B[i_1,...,i_r]\xi^{i_k-1})$
\end{proposition}

For convenience, we denote the reversal of GCA(A,B) on the k-th dimension as E1(k,A,B), the scaling of A by $\xi^k$ as $\xi^kA$, and the positional scaling of GCA(A,B) on the k-th dimension as E5(k,A,B), and we can do it multiple times, denoted as $E5^q$(k,A,B). 

There are two possible methods to generate additional GCAs. The first is to apply Proposition~\ref{proposition1} to the 2-phase GCS pair $A,B$ and the H-phase GCS pair $E,F$, and Proposition~\ref{proposition2} to the $r$-dimensional H-phase GCA pair $C,D$ during each application of Corollary~\ref{corollary:1}. The second is to apply Proposition~\ref{proposition2} to the GCA pair $G,H$ obtained after each application of Corollary~\ref{corollary:1}. However, applying all of the above equivalence operations produces many duplicated GCAs. The duplicated cases are summarized in the following proposition.

First we need the following property of P and Q.
\begin{lemma}\cite[Lemma 4]{Craigen}\cite[Proposition~10]{Du}
\label{lemma1}
Let A,B be 2-phase GCS, and P,Q defined as:\\
$P = \frac{1}{4}\bigl[ A + B + (B^{*} - A^{*}) \bigr],\
Q = \frac{1}{4}\bigl[ A + B - (B^{*} - A^{*}) \bigr]$\\
then P,Q,$P^*,Q^*$ are disjoint (0,$\pm1$) sequences, {P[i],Q[i],$P^*[i],Q^*[i]$} consist of three zeros and one $\pm1$ for any i.
\end{lemma}

\begin{proposition}\label{proposition3}
Let a construction of Corollary~\ref{corollary:1} be denoted as (G,H)=M(A,B,C,D,E,F), where A,B is a nontrivial 2-phase GCS of length $s_{r+1}$, C,D is a H-phase r-dimensional GCA of size $s_1\times\cdot\cdot\cdot\times s_r$, E,F is a H-phase GCS of length $s_{r+2}$, G,H is a H-phase r+2 dimensional GCA of size  $s_1\times\cdot\cdot\cdot\times s_r\times s_{r+1}\times s_{r+2}$. P,Q,X,Y,G,H are defined with ~\eqref{eqthm2}.
\begin{equation*}
\begin{aligned}
P &= \frac{1}{4}\bigl[ A + B + (B^{*} - A^{*}) \bigr],\
Q = \frac{1}{4}\bigl[ A + B - (B^{*} - A^{*}) \bigr],\\
X &= P \otimes C + Q \otimes D,\
Y = Q^{*} \otimes C - P^{*} \otimes D,\\
G &= X \otimes E + Y \otimes F,\
H = Y^{*} \otimes E - X^{*} \otimes F.
\end{aligned}
\end{equation*}
Then the following hold:\\  
T1. Reversal on the k-th dimension of (G,H), where $k\in 1,2,...r$, is equivalent to Reversal on the k-th dimension of (C,D); Reversal on the (r+1)th dimension of (G,H) is equivalent to Reversal on (A,B); Reversal on the (r+2)th dimension of (G,H) is equivalent to Reversal on (E,F).\\
E1(k,G,H)=M(A,B,E1(k,C,D),E,F), E1(r+1,G,H)=M(E1(A,B),C,D,E,F), E1(r+2,G,H)=M(A,B,C,D,E1(E,F))\\
T2. Positional Scaling on the kth dimension of (G,H), where $k\in 1,2,...r$, is equivalent to Positional Scaling on the kth dimension of (C,D); Positional Scaling on the (r+2)-th dimension of (G,H) is equivalent to Positional Scaling on (E,F)\\
E5(k,G,H)=(G1,H1)=M(A,B,E5(k,C,D),E,F), E5(r+2,G,H)=(G1,H1)=M(A,B,C,D,E5(E,F)),\\
T3. Swap (G,H): (H,G)=(G1,H1)=M($A^*,-B^*,D^*,C^*,E,-F$)\\
T4. Conjugate Reverse H: (G,$H^*$)$\simeq$(G1,H1)=M(A,$B^*$,E,F,C,D), where $\simeq$ means the matrix is equivalent after reordering\\
T5. Scale A: (G1,H1)=M(-A,B,-D,-C,F,E)=(G,-H)\\
T6. Swap (A,B): (G1,H1)=M(B,A,D,C,E,-F)=(G,-H)\\
T7. Reverse both A,B: (G1,H1)=M($A^*,B^*$,D,-C,-F,E)=(G,H)\\
T8. 
Positional Scaling (A,B): (G1,H1)=M(E5(A,B),D,C,E,F)=$E5^{H/2}(r+1,G,-H)$\\
T9. Scale C,D,E,F: Let (G2,H2) = M ($A,B,\xi^{k_1}C,\xi^{k_2}D,\xi^{k_3}E,\xi^{k_4}F$), (G1,H1) = M ($A,B,C,\xi^{k_2-k_1}D,E,\xi^{k_4-k_3}F$), \\we have G2 = $\xi^{k_1+k_3}$ G1, H2 = $\xi^{k_3-k_1}$ H1\\   
T10. Scale C,D,E,F for length 2 (A,B): When (A,B) is of length 2, we may assume A=(1,1), B=(1,-1), then\\(G1,H1)=M(A,B,C,D,E,$\xi^kF$)=$E5^k$(r+1,G,H), (G2,H2)=M(A,B,C,$\xi^kD$,E,F)=$E5^{-k}(r+1,\xi^kG,H)$
\end{proposition}
\begin{proof}
T1. We unfold (G,H) on the (r+1)th dimension,\\
G(:,...,:,i,:)=$P[i]\cdot C\otimes E+Q[i]\cdot D\otimes E+Q^*[i]\cdot C\otimes F-P^*[i]\cdot D\otimes F$\\
H(:,...,:,i,:)=$Q[i]\cdot C^*\otimes E-P[i]\cdot D^*\otimes E-P^*[i]\cdot C^*\otimes F-Q^*[i]\cdot D^*\otimes F$\\
Since P and Q is not trivial only on the (r+1)-th dimension and E and F is not trivial only on the (r+2)-th dimension, C,D is trivial on (r+1)-th and (r+2)-th dimension, we have\\
E1(k,G,H)=M(A,B,E1(k,C,D),E,F),E1(r+2,G,H)=M(A,B,C,D,E1(E,F))\\
Let $A1=A^*,B1=B^*$,$P1=\frac{1}{4}\bigl[A1+B1 + (B1^*-A1^*) \bigr]=P^*,Q1=\frac{1}{4}\bigl[A1+B1 - (B1^*-A1^*) \bigr]=Q^*$\\
Since lemma \ref{lemma1},$P[i],Q[i],P^*[i],Q^*[i]$ consists of three 0 and one $\pm1$, \\
we denote (G1,H1)=E1(r+1,G,H), we unfold (G1,H1) on the (r+1)th dimension,\\
G1(:,...,:,i,:)=G(:,...,:,$s_{r+1}+1-i$,:)=$P[s_{r+1}+1-i]\cdot C\otimes E+Q[s_{r+1}+1-i]\cdot D\otimes E+Q^*[s_{r+1}+1-i]\cdot C\otimes F-P^*[s_{r+1}+1-i]\cdot D\otimes F$=$P1[i]\cdot C\otimes E+Q1[i]\cdot D\otimes E+Q1^*[i]\cdot C\otimes F-P1^*[i]\cdot D\otimes F$,\\
similarly, H1(:,...,:,i,:)=$Q1[i]\cdot C^*\otimes E-P1[i]\cdot D^*\otimes E-P1^*[i]\cdot C^*\otimes F-Q1^*[i]\cdot D^*\otimes F$\\
compare (G1,H1) and (G,H), we have E1(r+1,G,H)=M(E1(A,B),C,D,E,F).\\
T2. The prove is the similer to that of T1.\\
T3. (G1,H1)=M($A^*,-B^*,D^*,C^*,E,-F$),then we have\\
$P1 = \frac{1}{4}\bigl[ A^* - B^* + (-B+A) \bigr]=-P,\\
Q1 = \frac{1}{4}\bigl[ A^* - B^* - (-B+A) \bigr]==Q,\\  
X1=P1\otimes(D^*)+Q1\otimes(C^*)=-P\otimes D^*+Q\otimes C^*=Y^*,\\
Y1=Q1^*\otimes(D^*)-P1^*\otimes(C^*)=Q^*\otimes D^*+P^*\otimes C^*=X^*,\\
G1=X1\otimes E+Y1\otimes (-F)=Y^*\otimes E-X^*\otimes F=H,\\
H1=Y1^*\otimes E-X1^*\otimes (-F)=X^*\otimes E+Y^*\otimes F=G.$\\
T4. $P1 = \frac{1}{4}\bigl[ A+B^* + (B-A^*) \bigr]=P,\\
Q1 = \frac{1}{4}\bigl[ A+ B^* - (B-A^*) \bigr]==Q^*,$\\  
We unfold (G,H) on the (r+1)th dimension,\\
G(:,...,:,i,:)=$P[i]\cdot C\otimes E+Q[i]\cdot D\otimes E+Q^*[i]\cdot C\otimes F-P^*[i]\cdot D\otimes F$\\
H(:,...,:,i,:)=$Q[i]\cdot C^*\otimes E-P[i]\cdot D^*\otimes E-P^*[i]\cdot C^*\otimes F-Q^*[i]\cdot D^*\otimes F$\\
then $H^*(:,...,:,i,:)$=$Q^*[i]\cdot C\otimes E^*-P^*[i]\cdot D\otimes E^*-P[i]\cdot C\otimes F^*-Q[i]\cdot D\otimes F^*$\\
For (G1,H1)=M(A,$B^*$,E,F,C,D), 
$P1 = \frac{1}{4}\bigl[ A + B^* + (B-A^*) \bigr]=P,\\
Q1 = \frac{1}{4}\bigl[ A + B^* - (B-A^*) \bigr]=Q^*,\\
X1=P1\otimes E+Q1\otimes F=P\otimes E+Q^*\otimes F,\\
Y1=Q1^*\otimes E-P1^*\otimes F=Q\otimes E+P^*\otimes F,\\
G1=X1\otimes C+Y1\otimes D,\\
H1=Y1^*\otimes C-X1^*\otimes D.$\\
we also unfold (G1,H1) on the (r+1)th dimension, \\
G1(:,...,:,i,:)=$P[i]\cdot E\otimes C+Q[i]\cdot F\otimes C+Q^*[i]\cdot E\otimes D-P^*[i]\cdot F\otimes D$\\
H1(:,...,:,i,:)=$Q^*[i]\cdot E^*\otimes C-P[i]\cdot F^*\otimes C-P^*[i]\cdot E^*\otimes D-Q[i]\cdot F^*\otimes D$\\
Compare (G,$H^*$) and (G1,H1), we have (G,$H^*$)$\simeq$(G1,H1).\\
T5. (G1,H1)=(-A,B,-D,-C,F,E),then we have\\
$P1 = \frac{1}{4}\bigl[ -A + B + (B^{*} + A^{*}) \bigr]=\frac{1}{4}\bigl[ A^{*} + B^{*} + (B- A) \bigr]=P^*,\\
Q1 = \frac{1}{4}\bigl[ -A + B - (B^{*} + A^{*}) \bigr]=-\frac{1}{4}\bigl[ A^{*} + B^{*} - (B- A) \bigr]=-Q*,\\  
X1=P1\otimes(-D)+Q1\otimes(-C)=-P^*\otimes D+Q^*\otimes C=Y,\\
Y1==Q1^*\otimes(-D)-P1^*\otimes(-C)=Q\otimes D+P\otimes C=X,\\
G1=X1\otimes F+Y1\otimes E=Y\otimes F+X\otimes E=G,\\
H1=Y1^*\otimes F-X1^*\otimes E=X^*\otimes F-Y^*\otimes E=-H.$\\
T6. (G1,H1)=M(B,A,D,C,E,-F),then we have\\
$P1=\frac{1}{4}\bigl[ B + A + (A^{*} - B^{*}) \bigr]=Q,\\
Q1=\frac{1}{4}\bigl[ B + A - (A^{*} - B^{*}) \bigr]=P,\\
X1=P1\otimes D+Q1\otimes C=Q\otimes D+P\otimes C=X,\\
Y1=Q1^*\otimes D-P1^*\otimes C=P^*\otimes D-Q^*\otimes C=-Y,\\
G1=X1\otimes E+Y1\otimes (-F)=X\otimes E+Y\otimes F=G,\\
H1=Y1^*\otimes E-X1^*\otimes (-F)=-Y*\otimes E+X^*\otimes F=-H.$\\
T7. (G1,H1)=M($A^*, B^*$,D,-C,-F,E),then we have\\
$P1=\frac{1}{4}\bigl[ A^* + B^* + (B - A) \bigr]=P^*,\\
Q1=\frac{1}{4}\bigl[ A^* + B^* - (B - A) \bigr]=Q^*,\\
X1=P1\otimes D+Q1\otimes(-C)=P^*\otimes D-Q^*\otimes C=-Y,\\
Y1==Q1^*\otimes D-P1^*\otimes(-C)=Q\otimes D+P\otimes C=X,\\
G1=X1\otimes (-F)+Y1\otimes E=Y\otimes F+X\otimes E=G,\\
H1=Y1^*\otimes (-F)-X1^*\otimes E=-X^*\otimes F+Y^*\otimes E=H.$\\
T8. (A,B) is a nontrivial 2-phase GCS, $s_{r+1}$ is an even integer.
When (A1,B1)=E5(A,B), we have \\A1[i]=$(-1)^{i-1}$A[i], B1[i]=$(-1)^{i-1}$B[i], \\$A1^*[i]=(-1)^{s_{r+1}-i}A[s_{r+1}+1-i]$=$(-1)^iA^*[i]$, \\$B1^*[i]=(-1)^{s_{r+1}-i}B[s_{r+1}+1-i]$=$(-1)^iB^*[i]$,\\
$P1[i]=\frac{1}{4}\bigl[ A1[i] + B1[i] + (B1^{*}[i] - A1^{*}[i]) \bigr]=\frac{1}{4}(-1)^{i-1}\bigl[ A1[i] + B1[i] - (B1^{*}[i] - A1^{*}[i])\bigr]=(-1)^{i-1}Q[i],$\\
$Q1[i]=(-1)^{i-1}P[i],\\P1^*[i]=P1[s_{r+1}+1-i]=(-1)^{s_{r+1}-i}Q[s_{r+1}+1-i]=-(-1)^{i-1}Q^*[i],\\Q1^*[i]=-(-1)^{i-1}P^*[i],\\$
we unfold (G1,H1)=M(A1,B1,D,C,E,F) on the (r+1)th dimension,\\
G1(:,...,:,i,:)=$(-1)^{i-1}(Q[i]\cdot D\otimes E+P[i]\cdot C\otimes E-P^*[i]\cdot D\otimes F+Q^*[i]\cdot C\otimes F)$\\
H1(:,...,:,i,:)=$(-1)^{i-1}(P[i]\cdot D^*\otimes E-Q[i]\cdot C^*\otimes E+Q^*[i]\cdot D^*\otimes F+P^*[i]\cdot C^*\otimes F)$\\
we also unfold (G,H)=M(A,B,C,D,E,F) on the (r+1)th dimension,\\
G(:,...,:,i,:)=$P[i]\cdot C\otimes E+Q[i]\cdot D\otimes E+Q^*[i]\cdot C\otimes F-P^*[i]\cdot D\otimes F$\\
H(:,...,:,i,:)=$Q[i]\cdot C^*\otimes E-P[i]\cdot D^*\otimes E-P^*[i]\cdot C^*\otimes F-Q^*[i]\cdot D^*\otimes F$\\
Compare (G1,H1) and (G,H), we have (G1,H1)=$E5^{H/2}(r+1,G,-H)$\\
T9.(G1,H1) = M ($A,B,C,\xi^{k_2-k_1}D,E,\xi^{k_4-k_3}F$), (G2,H2) = M ($A,B,\xi^{k_1}C,\xi^{k_2}D,\xi^{k_3}E,\xi^{k_4}F$), then\\
X2=$P\otimes (\xi^{k_1}C)+Q\otimes (\xi^{k_2}D)$=$\xi^{k_1}$X1,\\
Y2=$Q^*\otimes (\xi^{k_1}C)-P^*\otimes (\xi^{k_2}D)$=$\xi^{k_1}$Y1,\\
G2=$(\xi^{k_1}X1)\otimes(\xi^{k_3}E)+(\xi^{k_1}Y1)\otimes(\xi^{k_4}F)$=$\xi^{k_1+k_3}(X1\otimes E+Y1\otimes (\xi^{k_4-k_3})F)$=$\xi^{k_1+k_3}$G1,\\
H2=$(\xi^{k_1}Y1)^*\otimes(\xi^{k_3}E)-(\xi^{k_1}X1)^*\otimes(\xi^{k_4}F)$=$\xi^{k_3-k_1}(X1\otimes E+Y1\otimes (\xi^{k_4-k_3})F)$=$\xi^{k_3-k_1}$H1,\\
T10. There are 8 distinct length 2 2-phase GCS, which can be obtained from A=[1,1] and B=[1,-1] by scale A,B respectively and swap (A,B), due to the previous discussion, scale and swap the 2-phase GCS do not generate additional GCA, therefore we only need to discuss one case.\\
When A=[1,1] and B=[1,-1], we have P=[0,0] and Q=[1,0], then\\
G=$(D\otimes E)|(C\otimes F)$, H=$(C^*\otimes E)|(-D^*\otimes F)$, where $|$ denotes concatenating two arrays in the (r+1)th dimension.\\
G1=$(D\otimes E)|(C\otimes (\xi^kF))$=$(D\otimes E)|\xi^k(C\otimes F)$,\\
H1=$(C^*\otimes E)|(-D^*\otimes (\xi^kF))$=$(C^*\otimes E)|\xi^k(-D^*\otimes F)$,\\
(G1,H1)=$E5^k$(r+1,G,H).\\
G2=$((\xi^kD)\otimes E)|(C\otimes F)$=$\xi^k((D\otimes E)|\xi^{-k}(C\otimes F))$,\\
H2=$(C^*\otimes E)|(-(\xi^kD)^*\otimes F)$=$(C^*\otimes E)|\xi^{-k}(-D^*\otimes F)$,\\(G2,H2)=$E5^{-k}(r+1,\xi^kG,H)$
\end{proof}

By Proposition ~\ref{proposition3}, we can can summarize the remaining equivalence operations needed to avoid duplicated GCAs in the construction of Theorem~\ref{Theorem:3}.

\begin{proposition}\label{proposition4}
For the construction of Theorem \ref{Theorem:3}, the remaining equivalence operations are as follows:\\
1. For the r nontrivial 2-phase GCSs and r+1 H-phase GCSs, consider all permutations of 2-phase GCSs and H-phase GCSs separately as entries of Theorem \ref{Theorem:3}.\\
2. For each H-phase entry, apply conjugate reverse to either sequence, reverse both sequences, or swap the two sequences, thereby generating at most 16 pairs. For each 2-phase entries, reverse one of the sequences to generate at most 2 pairs. Some of these pairs may coincide, which will be discussed in the next section.\\
3. At each step of the construction using Corollary \ref{corollary:1}, if the 2-phase GCS is not of length 2, scale D and F with $\xi^k,k\in(0,1,...,H-1)$ respectively. This generates $H^2$ additional GCAs each time.\\
4. After the construction in Theorem \ref{Theorem:3}, scale each array in $(G,H)$ by $\xi^k$, $k\in\{0,1,\ldots,H-1\}$, and apply positional scaling on each dimension. This generates $H^{2r+3}$ additional GCAs.
\end{proposition}

\begin{proof}
We separate the construction of Theorem \ref{Theorem:3} into r applications of Corollary \ref{corollary:1}. First, consider one applications of Corollary \ref{corollary:1}, denoted as (G,H)=M(A,B,C,D,E,F).

For equivalent operations on (G,H): By T1, reversal on the k-th dimension of (G,H) is equivalent to reversal on (C,D),(A,B),(E,F) with the corresponding dimension. By T3, swapping (G,H) is equivalent to scaling and conjugate reversing of (A,B,C,D,E,F). By T4 together with T3, conjugate reverse of (G,H) is equivalent to swapping the order of (C,D),(E,F) and reversing (A,B). Therefore, for $(G,H)$, it suffices to consider scaling each array and positional scaling on each dimension.

For equivalent operations with (A,B): By T5-T8, scaling, swapping, reversing both sequences, and positional scaling do not generate additional GCAs. Thus the only remaining operation is reversing one of the two sequences.

For equivalent operations with (C,D) and (E,F): By T9, scaling C,D,E,F is equivalent to scaling D and F. By T10, when (A,B) has length 2, scaling D,F also does not generate additional GCAs. Therefore, for scaling C,D,E,F, it is sufficient to scale $D$ and $F$ by $\xi^k$, $k\in\{0,1,\ldots,H-1\}$, respectively, when $(A,B)$ is not of length 2. By T2, positional scaling on (C,D) and (E,F) is equivalent to positional scaling on (G,H), which has already been considered. For (C,D) and (E,F) , the remaining operations are reversal, swap, and conjugate reversal.

Finally, returning to the construction of Theorem \ref{Theorem:3}, since each GCA (C,D) comes from a previous application of Corollary \ref{corollary:1}, the conclusions above show that it is enough to consider scaling (C,D) and positional scaling. Again, by T2, positional scaling on (C,D) is equivalent to positional scaling on the final (G,H), and by T9 and T10, we only need to scale D when (A,B) is not of length 2.

Thus, all equivalent operations not listed in Proposition~\ref{proposition4} have been excluded.
\end{proof}

In the construction of \cite{Fiedler2008} and \cite{Gibson}, such an equivalent operation is called "affine offset" and corresponds to the fourth operation of  Proposition \ref{proposition4}. The seed pairs in \cite{Gibson} correspond to the second operation in Proposition \ref{proposition4}. As noted above, Theorem~\ref{Theorem:1} is a special case of Theorem~\ref{Theorem:3} in which all 2-phase GCS entries have length 2. In the more general case, three additional operations must be considered:
\begin{enumerate}
    \item a full permutation of the 2-phase GCSs and H-phase GCSs, respectively;
    \item reversal of both H-phase seed pairs;
    \item scaling D and F, respectively, when the 2-phase GCS is not of length 2 during the construction in Corollary~\ref{corollary:1}.
\end{enumerate}

\subsection{Projections from GCA to GCS}

The projections from GCA to GCS follows exactly the same method as in \cite{Fiedler20081}\cite{Fiedler2008} and \cite{Gibson}. The key idea is that an $r$-dimensional GCA can be projected to an $(r-1)$-dimensional GCA, as stated in the following theorem.
\label{theorem4}
\begin{theorem}\cite[Theorem~11]{Jedwab2007}
Suppose that A and B form an r-dimensional Golay array pair over an alphabet S,with $r\ge2$.Then, for distinct $k,l\in \{1,...,r\}$,$\Psi_{k,l}(A)$ and $\Psi_{k,l}(B)$form an (r-1)-dimensional Golay array pair over S.
\end{theorem}

We apply Theorem~\ref{theorem4} successively until an r-dimensional array is projected to a 1-dimensional sequence. In order to keep track of the projection mappings, we represent the array indices by vertices $1,\ldots,r$ of a directed graph and represent successive projection mappings by arcs between vertices. The projected array corresponding to a given graph does not depend on the order in which arcs are added\cite[Proposition~2]{Fiedler2008}. In particular, the sequence obtained after applying r-1 successive projection mappings to an r-dimensional array is completely described by a directed path of the form\\
$\overset{\sigma(1)}{\bullet} \to\overset{\sigma(2)}{\bullet}\to\cdot\cdot\cdot\to\overset{\sigma(r)}{\bullet}$\\
for some permutation $\sigma$ of $\{1,\ldots,r\}$.
Therefore, the total number of distinct projections of an r-dimensional Golay is $r!$. 

\section{Applications of the new construction}
In this section, we first describe the seed pairs used in the construction. Using the five 2-phase seeds of lengths 2,10,10,20,26 and five 4-phase seeds of lengths 3,5,8,11,13, we can construct all known 2-phase and 4-phase Golay sequences in a uniform manner. For lengths not covered by computer search, we also obtain lower bounds on the numbers of 2-phase GCSs of lengths $2^a10^b26^c$ and 4-phase GCSs of lengths $2^{a+u}3^b5^c11^d13^e$.
\subsection{Primitive seed pairs}
\label{subsec1}
Primitive seed pairs are GCSs that cannot be generated from shorter GCSs. Although other GCSs could also be used as entries of Theorem~\ref{Theorem:3}, they are already generated from the primitive seeds by definition and therefore do not yield additional GCSs. Moreover, using nonprimitive seeds typically reduces the number of dimensions in the resulting GCAs, which in turn reduces the number of GCSs obtained after equivalence operations and projections. For this reason, we restrict attention to primitive seed pairs.

In our construction, we use the five 2-phase GCS seeds of length 2,10,10,20,26 from \cite{Borwein} and the five 4-phase GCS seeds of length 3,5,8,11,13 from \cite{Gibson} and \cite{Fiedler2008}, together with the trivial seeds of length 1. For the 2-phase seed, we use length 2 instead of length 1 as in \cite{Borwein} to satisfy the requirements of Theorem~\ref{Theorem:3}. The seeds are listed below:
\paragraph{2-phase primitive seed pairs}
\begin{equation*}
\begin{split}
&A_2=[1,1],B_2=[1,-1], \\
&A_{10-1}=[1, 1, -1, 1, -1, 1, -1, -1, 1, 1],B_{10-1}=[1, 1, -1, 1,  1, 1,  1,  1, -1, -1],\\
&A_{10-2}=[1, 1,  1, 1,  1, -1,  1, -1, -1, 1],B_{10-2}=[1, 1, -1,-1,  1,  1,  1, -1,  1, -1],\\
&A_{20}=[1,1,1,1,-1,1,-1,-1,-1,1,1,-1,-1,1,1,-1,1,-1,-1,1],\\
&B_{20}=[1,1,1,1,-1,1,1,1,1,1,-1,-1,-1,1,-1,1,-1,1,1,-1],\\
&A_{26}=[1, 1, 1, 1, -1, 1, 1, -1, -1, 1, -1, 1, -1, 1, -1, -1, 1, -1, 1, 1, 1, -1, -1, 1, 1, 1],\\
&B_{26}=[1, 1, 1, 1, -1, 1, 1, -1, -1, 1, -1, 1, 1, 1, 1, 1, -1, 1, -1, -1, -1, 1, 1, -1, -1, -1],\\
\end{split}
\end{equation*}

\paragraph{4-phase primitive seed pairs}
\begin{equation*}
\begin{split}
&A_3=[1,1,-1],B_3=[1,i,1], \\
&A_5=[1,1,1,-i,i],B_5=[1,i,-1,1,-i],\\
&A_8=[1,1,1,-1,1,1,-1,1],B_8=[1,i,i,-1,1,-i,-i,-1],\\
&A_{11}=[1,1,1,i,-1,1,i,-i,i,1,-1],B_{11}=[1,i,-1,-1,-1,i,i,1,-i,i,1],\\
&A_{13}=[1,1,1,i,-1,1,1,-i,1,-1,1,-i,i],B_{13}=[1,i,-1,-1,-1,i,-1,1,1,-i,-1,1,-i],\\
\end{split}
\end{equation*}

We use the primitive seed pairs($A_s,B_s$) to define the set of ordered sequence pairs. We also observe the following equivalences on some of the seeds:

\begin{proposition}\label{proposition5}
The following equivalences hold:\\
1.Conjugate Reverse of length 3 seed:\\
($A_3^*,B_3$)=$E5(-B_3,A_3$)\\
($A_3,B_3^*$)=$E5^3(B_3,A_3$)\\
($A_3^*,B_3^*$)=$E5^2(A_3,B_3$)\\
2.Conjugate Reverse of length 2 seed:\\
($A_2,B_2^*$)=$(A_2,-B_2$)\\
3.Reverse both sequence of all the 4-phase seeds:\\
E1($A_3,B_3$)=($A_3^*,B_3$)\\
E1($A_5,B_5$)=$E5^3(B_5^*,A_5^*)$\\
E1($A_8,B_8$)=($A_8^*,-B_8$)\\
E1($A_{11},B_{11}$)=$E5^3(-B_{11}^*,-A_{11}^*)$\\
E1($A_{13},B_{13}$)=$E5^3(B_{13}^*,A_{13}^*)$\\
\end{proposition}

Proposition~\ref{proposition5} follows directly from the definitions of the equivalence operations. By Proposition~\ref{proposition4}, each 4-phase seed yields at most 16 equivalent pairs, while each 2-phase seed yields at most 2. Proposition~\ref{proposition5} eliminates the duplicated cases. The resulting distinct seed pairs are

\begin{equation*}
P_S=\left\{\begin{matrix}
 \left \{ (A_S,B_S),(B_S,A_S) \right \}&for \quad s=3 \\
 \left \{ (A_S,B_S),(A_S,B_S^*),(A_S^*,B_S),(A_S^*,B_S^*), 
(B_S,A_S),(B_S,A_S^*),(B_S^*,A_S),(B_S^*,A_S^*) \right \}&for \quad s=5,8,11,13 \\
 \left \{ (A_S,B_S) \right \}&for \quad s=2 \\
 \left \{ (A_{S-1},B_{S-1}),(A_{S-1},B^*_{S-1}),(A_{S-2},B_{S-2}),(A_{S-2},B^*_{S-2}) \right \}&for \quad s=10\\
 \left \{ (A_S,B_S),(A_S,B^*_S) \right \}&for \quad s=20,26\\
\end{matrix}\right. 
\end{equation*}

\subsection{2-phase case}
All known 2-phase GCSs can be constructed from the 5 2-phase primitive seed pairs.
We first consider the lengths of the form $n=2^a10^b20^c26^d$.
\begin{lemma}
For a length n written as $2^a10^b20^c26^d$, the construction yields
\begin{equation}\label{equation:5}
N=\frac{((a+b+c+d)!)^2}{a!b!c!d!}2^{a+5b+4c+4d+2}
\end{equation}
2-phase GCSs.
\end{lemma}
\begin{proof}
We use a+b+c+d 2-phase seeds and a+b+c+d+1 trivial seeds of length 1 as the inputs of~\ref{Theorem:3}. Here, the nontrivial seeds have lengths $2,10,20,$ and $26$, with multiplicities $a,b,c,$ and $d$, respectively.
The construction proceeds in four steps.\\
\emph{Step 1: Seed permutation.}
Permute the nontrivial seeds. The number of distinct permutations is
\[
\frac{(a+b+c+d)!}{a!b!c!d!}.
\]

\emph{Step 2: Seed choices.}
For each permutation, choose the seed pairs from Section~\ref{subsec1} and build an $(a+b+c+d)$-dimensional GCA. The number of choices is the product of:
\begin{itemize}
    \item the seed-pair choices from Section~\ref{subsec1}, contributing $1^a4^b2^c2^d$;
    \item the independent scalings of $D$ and $F$ in Corollary~\ref{corollary:1} for the seeds of lengths $10,20,$ and $26$, contributing $(2\times2)^{b+c+d}$.
\end{itemize}

\emph{Step 3: Equivalence operations.}
For each resulting pair $(G,H)$, apply $\mathrm{E4}$ to $G$ and $H$, and apply $\mathrm{E5}$ on each dimension. This gives $2^{a+b+c+d+2}$ GCAs for each choice in Step 2.

\emph{Step 4: Projection.}
Project each GCA to a 1-dimensional GCS. By successive projections, this yields $(a+b+c+d)!$ GCSs.

Multiplying the contributions of the four steps gives \eqref{equation:5}.
\end{proof}

\begin{theorem}\label{theorem5}
For length $n=2^a10^b26^c$, where $a,b,c\ge0$, the number of 2-phase GCSs constructed is
\begin{equation}
   N= \sum_{k=0}^{min(a,b)} \frac{((a+b+c-k)!)^2}{(a-k)!(b-k)!k!c!}2^{a+5b+4c-2k+2} 
\end{equation}
\end{theorem}
\begin{proof}
Write $n=2^a10^b26^c=2^{a-k}10^{b-k}20^{k}26^c$, where $0\le k\le min(a,b)$. For each admissible k, apply \eqref{equation:5} with parameters $(a-k,b-k,k,c)$. Summing up all k gives the stated formula.
\end{proof}

When we set b=c=0 in Theorem 5, then for length $n=2^a$,we have $N=a!2^{a+2}$, which agrees with \cite[Theorem 4.5]{Borwein}. When we set one of b,c,d in lemma 2 to 1 and the others to 0, then for length $n=2^ak$, where k is one of the seeds of length 10,10,20,26, we have $N=(a+1)(a+1)!2^{a+6}$, which agrees with \cite[Theorem 4.6]{Borwein}. The construction of Theorem 5 covers all the known 2-phase GCS lengths. Table~\ref{tab:table1} lists the number of 2-phase GCSs constructed with Theorem 5 for $n\le5000$.

\begin{table}[!t]
\caption{Number of 2-phase GCSs of length $\le5000$\label{tab:table1}}
\centering
\begin{tabular}{|c|c||c|c||c|c||c|c|}
\hline
length & \# of pairs & length & \# of pairs & length & \# of pairs& length & \# of pairs\\
\hline
1 & 4 & 80 & 102912 & 512 & 743178240& 1664& 144506880\\
\hline
2 & 8 & 100 & 8192 & 520 & 151552&  2000& 26345472\\
\hline
4 & 32 & 104 & 4608 & 640 & 297861120& 2048&326998425600\\
\hline
8 & 192 & 128 & 2580480 & 676 & 2048& 2080 & 40501248\\
\hline
10 & 128 & 160 & 1277952 & 800 & 41717760& 2560&109660078080\\
\hline
16 & 1536 & 200 & 155648 & 832 & 8847360&2600&1179648\\
\hline
20 & 1088 & 208 & 49152 & 1000 & 786432&2704&589824\\
\hline
26 & 64 & 256 & 41287680 & 1024 & 14863564800&3200&14590279680\\
\hline
32 & 15360 & 260 & 8192 & 1040 & 2433024&3328&2642411520\\
\hline
40 & 9728 & 320 & 18309120 & 1280 & 5429329920&4000&667484160\\
\hline
52 & 512 & 400 & 2508800 & 1352&36864& 4096&7847962214400\\
\hline
64 & 184320 & 416 & 614400 & 1600&747700224&4160&727449600\\
\hline
\end{tabular}
\end{table}

\subsection{4-phase case}
\begin{lemma}
\label{lemma2}
For a length n written as $2^{k_2}10^{k_{10}}20^{k_{20}}26^{k_{26}}3^{k_3}5^{k_5}8^{k_8}11^{k_{11}}13^{k_{13}}$, the construction yields $N$ 4-phase GCSs, where
\begin{equation}
\label{eq4}
\begin{aligned}
  & N= \frac{S_1!}{k_2!k_{10}!k_{20}!k_{26}!}\frac{(S_1+1)!}{k_3!k_5!k_8!k_{11}!k_{13}!(S_1+1-S_2)!}2^{2k_2+8k_{10}+7k_{20}+7k_{26}+3k_3+5k_5+5k_8+5k_{11}+5k_{13}+4}(S_1+S_2)!\\
\end{aligned}
\end{equation}
with
\[
S_1=k_2+k_{10}+k_{20}+k_{26},\qquad
S_2=k_3+k_5+k_8+k_{11}+k_{13}.
\]
\end{lemma}

\begin{proof}
We use $S_1+1$ 2-phase GCSs and $S_1+1$ 4-phase GCSs as inputs to Theorem~\ref{Theorem:3}. The 2-phase inputs consist of $k_2,k_{10},k_{20},k_{26}$ seeds of lengths $2,10,20,$ and $26$, respectively. The 4-phase inputs consist of $k_3,k_5,k_8,k_{11},k_{13}$ seeds of lengths $3,5,8,11,$ and $13$, together with $S_1+1-S_2$ trivial seeds of length $1$.

The construction proceeds in four steps:

\emph{Step 1: Seed permutation.}
Permute the 2-phase seeds and the 4-phase seeds independently. The number of such permutations is
\[
\frac{S_1!}{k_2!k_{10}!k_{20}!k_{26}!}
\frac{(S_1+1)!}{k_3!k_5!k_8!k_{11}!k_{13}!(S_1+1-S_2)!}.
\]

\emph{Step 2: Seed-pair choices.}
For each permutation, choose the seed pairs from Section~\ref{subsec1} to form an $(S_1+S_2)$-dimensional GCA. The number of choices has two contributions:
\begin{itemize}
    \item the Cartesian product of the seed entries, contributing
    \[
    4^{k_{10}}2^{k_{20}}2^{k_{26}}2^{k_3}8^{k_5}8^{k_8}8^{k_{11}}8^{k_{13}};
    \]
    \item for each 2-phase seed of length $10$, $20$, or $26$, the independent choices of scaling $D$ and $F$ in Corollary~\ref{corollary:1} by $i^k$, $k\in\{0,1,2,3\}$, contributing
    \[
    (4\times 4)^{k_{10}+k_{20}+k_{26}}.
    \]
\end{itemize}

\emph{Step 3: Equivalence operations.}
For each resulting pair $(G,H)$, apply $\mathrm{E4}$ to $G$ and $H$, and apply $\mathrm{E5}$ on each dimension. This gives $4^{S1+S2+2}$ GCAs for each choice in Step 2.

\emph{Step 4: Projection.}
Project each GCA in Step 3 to a 1-dimensional GCS. This yields $(S_1+S_2)!$ GCSs.

Multiplying the contributions of the four steps gives \eqref{eq4}.
\end{proof}

\begin{theorem}
\label{theorem6}
For length $n=2^{a+u}3^b5^c11^d13^e$, where integer $a,b,c,d,e,u\ge0$, $b+c+d+e\le a+2u+1,u\le c+e$ the construction yields $N$ 4-phase Golay complementary sequence pairs, where
\begin{equation}
\label{eq60}
\begin{aligned}
   N=&\sum_{k_{10}=0}^{c}\sum_{k_{20}=0}^{c-k_{10}}\sum_{k_{26}=0}^{e}\sum_{k_{8}=0}^{min(M1,M2)} \frac{(a+u-3k_8-k_{20})!}{(a+u-3k_8-k_{10}-2k_{20}-k_{26})!k_{10}!k_{20}!k_{26}!}\\&\cdot\frac{(a+u-3k_8-k_{20}+1)!}{b!(c-k_{10}-k_{20})!k_8!d!(e-k_{26})!(a+u+k_{10}+k_{26}+1-b-c-d-e-4k_8)!}\\&\cdot2^{2a+2u+3b+5c+5d+5e-k_8+k_{10}-2k_{20}+4}(a+u+b+c+d+e-2k_8-k_{10}-2k_{20}-k_{26})!\\
\end{aligned}  
\end{equation}
where
\[
M_1=\left\lfloor \frac{a+u-b-c-d-e+1+k_{10}+k_{26}}{4}\right\rfloor,\qquad
M_2=\left\lfloor \frac{a+u-k_{10}-2k_{20}-k_{26}}{3}\right\rfloor.
\]
\end{theorem}
\begin{proof}
For the given length n, we enumerate all admissible values of $k_2,k_{10},k_{20},k_{26},k_3,k_5,k_8,k_{11},k_{13}$. For each admissible choice, Lemma~\ref{lemma2} gives the number of generated GCSs, and the total is obtained by summing over all such choices.

The constraint of $k_2,k_{10},k_{20},k_{26},k_3,k_5,k_8,k_{11},k_{13}$ for the given n is:
\begin{equation}
\label{eq61}
\begin{aligned}
&a+u=k_2+3k_8+k_{10}+2k_{20}+k_{26}\\
&b=k_3\\
&c=k_5+k_{10}+k_{20}\\
&d=k_{11}\\
&e=k_{13}+k_{26}\\
&k_2+k_{10}+k_{20}+k_{26}+1\ge k_3+k_5+k_8+k_{11}+k_{13}\\
&k_2,k_{10},k_{20},k_{26},k_3,k_5,k_8,k_{11},k_{13}\ge0
\end{aligned}   
\end{equation}

Using $k_{10},k_{20},k_{26},k_8$ as free variables, together with the given a,u,b,c,d,e, we have 

\begin{equation}
\label{eq62}
\begin{aligned}
k_2=a+u-3k_8-k_{10}-2k_{20}-k_{26},\qquad 
k_3=b,\qquad k_5=c-k_{10}-k_{20},\qquad k_{11}=d,\qquad k_{13}=e-k_{26},
\end{aligned}   
\end{equation}

Substituting \eqref{eq62} into \eqref{eq61} gives
\begin{equation}
\label{eq63}
\begin{aligned}
&a+u-3k_8-k_{10}-2k_{20}-k_{26}\ge0\\
&c-k_{10}-k_{20}\ge0\\
&e-k_{26}\ge0\\
&k_{10},k_{20},k_{26},k_{8}\ge0\\
&a+u-3k_8-k_{10}-2k_{20}-k_{26}+k_{10}+k_{20}+k_{26}+1\ge b+c-k_{10}-k_{20}+k_8+d+e-k_{26}
\end{aligned}   
\end{equation}

Hence, we have the range for $k_{10},k_{20},k_{26},k_{8}$:
\begin{equation}
\label{eq64}
\begin{aligned}
&0\le k_{10}\le c\\
&0\le k_{20}\le c-k_{10}\\
&0\le k_{26}\le e\\
&0\le k_8\le \frac{1}{3}(a+u-k_{10}-2k_{20}-k_{26})\\
&k_8\le \frac{1}{4}(a+u-b-c-d-e+1+k_{10}+k_{26})
\end{aligned}   
\end{equation}

For each $k_{10},k_{20},k_{26},k_{8}$ satisfying \eqref{eq64}, we use the corresponding values of $k_2,k_3,k_5,k_{11},k_{13}$ from \eqref{eq62} in Lemma~\ref{lemma2}. Summing over all admissible choices yields \eqref{eq60}.
\end{proof}

Theorem~\ref{theorem6} recovers several known results as special cases. 

If $k_i=0$ for all $i\neq 2$ in Lemma~\ref{lemma2}, then
\[
N=2^{2k_2+4}(k_2)!,
\]

which is the standard 4-phase Golay sequence count in \cite{Fiedler2006,Fiedler20081}. 

If $a+u=k$ and $b=c=d=e=0$ in Theorem~\ref{theorem6}, then 
\[
N=\sum_{k_8=0}^{min(\left \lfloor (k+1)/4 \right \rfloor, \left \lfloor k/3 \right \rfloor)} \frac{(k-3k_8+1)!}{k_8!(k+1-4k_8)!}2^{2k-k_8+4}(k-2k_8)! 
\]

which reduces to \cite[Corollary~16]{Fiedler2008}. 

If $k_i=0$ for all $i\neq 2,s$ and $k_2=m$, $k_s=c$ in Lemma~\ref{lemma2}, then for $n=2^ms^c,\qquad s\in\{3,5,11,13\}$,

we obtain
\[
N=\frac{(m+1)!}{c!(m+1-c)!}2^{2m+tc+4}(m+c)!,
\]

where $t=3$ for $s=3$ and $t=5$ for $s\in\{5,11,13\}$, in agreement with \cite[Proposition~9]{Gibson}. 

Theorem~\ref{theorem6} covers all known possible lengths of 4-phase GCSs, including many new lengths that is not reported before. Table~\ref{tab:table2} lists the numbers obtained from Theorem~\ref{theorem6} for $n\le500$.
\begin{table}[!t]
\caption{Number of 4-phase GCSs of length $\le500$\label{tab:table2}}
\centering
\begin{tabular}{|c|c||c|c||c|c||c|c|}
\hline
length & \# of pairs & length & \# of pairs & length & \# of pairs& length & \# of pairs\\
\hline
1 & 16 & 40 & 4358144 & 130 & 1572864 & 286 & 1048576\\
\hline
2 & 64 & 44 & 147456 & 132 & 9437184 & 288 & 79461089280\\
\hline
3 & 128 & 48 & 19857408 & 144 & 1887436800 & 300 & 2120220672\\
\hline
4 & 512 & 50 & 917504 & 150 & 12582912 & 338 & 655360\\
\hline
5 & 512 & 52 & 180224 & 156 & 11796480 & 320 & 94544855040\\
\hline
6 & 2048 & 60 & 14221312 & 160 & 2963537920 & 330 & 12582912\\
\hline
8 & 6656 & 64 & 50331648 & 176 & 79429632 & 338 & 655360\\
\hline
10 & 12288 & 66 & 196608 & 180 & 340525056 & 352 & 2302672896\\
\hline
11 & 512 & 72 & 47185920 & 192 & 18874368000 & 360 & 27254587392\\
\hline
12 & 36864 & 78 & 262144 & 200 & 1445986304 & 384 & 691556843520\\
\hline
13 & 512 & 80 & 105709568 & 208 & 92274688 & 390 & 18874368\\
\hline
16 & 106496 & 88 & 3145728 & 216 & 1509949440 & 396 & 188743680\\
\hline
18 & 24576 & 90 & 1572864 & 220 & 56885248 & 400 & 53175910400\\
\hline
20 & 215040 & 96 & 575668224 & 234 & 786432 & 416 & 2626682880\\
\hline
22 & 8192 & 100 & 40108032 & 240 & 20220739584 & 432 & 105696460800\\
\hline
24 & 786432 & 104 & 3735552 & 242 & 393216 & 440 & 2123366400\\
\hline
26 & 10240 & 108 & 15728640 & 250 & 25165824 & 468 & 264241152\\
\hline
30 & 327680 & 110 & 1310720 & 256 & 44562382848 & 480 & 820613480448\\
\hline
32 & 2113536 & 120 & 530841600 & 260 & 68419584 & 484 & 18874368\\
\hline
36 & 1179648 & 128 & 1400242176 & 264 & 377487360 & 500 & 4039114752\\
\hline

\end{tabular}
\end{table}

\subsection{A discussion}
Although the main focus of this manuscript is on 2-phase and 4-phase GCSs, the construction method in Section III can also be applied to H-phase GCSs. For example, Theorem~\ref{theorem5} can be extended to the H-phase case by replacing the phase 2 with H.

\begin{corollary}\label{corollary2}
For length $n=2^a10^b26^c$, where $a,b,c\ge0$, with only the 2-phase primitive seed pairs, the construction yields N H-phase golay complementary sequence pairs, where
\begin{equation}
   N= \sum_{k=0}^{min(a,b)} \frac{((a+b+c-k)!)^2}{(a-k)!(b-k)!k!c!}2^{2b-k+c}\cdot H^{a+3b+3c-k+2}  
\end{equation}
\end{corollary}

With suitable H-phase primitive seed pairs, a result similar to Theorem~\ref{theorem6} can be obtained. The existence of additional seed pairs arising from larger alphabets is an interesting open question. A better understanding of such seed pairs may lead to further constructions and more options for engineering applications. In addition, if new primitive 2-phase or 4-phase GCSs are discovered in the future, they can be incorporated into the constructions of Theorems~\ref{theorem5} and~\ref{theorem6}.

For comparison, Table~\ref{tab:table3} summarizes the sequence lengths and phases covered by the known constructions. Theorems~\ref{theorem5} and~\ref{theorem6} extend the existing construction methods to all known lengths for 2-phase and 4-phase GCSs in a uniform and extensible framework.
\begin{table}[!t]
\caption{Comparisons of GCSs obtained by different constructions\label{tab:table3}}
\centering
\begin{tabular}{|c|c||c|c||c|c||c|c|}
\hline
Construction & Sequence length & Phases \\
\hline
\cite[Theorem 4.5]{Borwein} &  $2^k$ & 2\\
\hline
\cite[Theorem 4.6]{Borwein} &  $2^kn$,n=10,20,26 & 2\\
\hline
Theorem \ref{theorem5} &  $2^a10^b26^c$ & 2\\
\hline
\cite[Corollary 5]{Davis} &  $2^k$(standard) & 4\\
\hline
\cite[Corollary 16]{Fiedler2008} &  $2^k$ & 4\\
\hline
\cite[Proposition 9]{Fiedler2008} &  $2^ms^c$,s=3,5,11,13 & 4\\
\hline
Theorem \ref{theorem6} &  $2^{a+u}3^b5^c11^d13^e$& 4\\
\hline
Corollary \ref{corollary2} &  $2^a10^b26^c$ & H\\
\hline
\end{tabular}
\end{table}

\section{Conclusion}
In this paper, we propose a general construction for the Golay complementary sequence. The construction extends the existing 3-stage-construction framework by allowing r+1 H-phase GCSs and r 2-phase GCSs as entries in the GCA construction. We also discuss the equivalence operations during the construction and principles to remove duplicated constructions.

As an application, the construction of 2-phase and 4-phase GCSs is discussed. All the known 2-phase and 4-phase GCSs can be generated in a uniform way. If new 2-phase or 4-phase seed pairs are found, they can be incorporated into the current framework. The construction method can also be adapted for H-phase GCSs. Closed-form expressions for the numbers of 2-phase and 4-phase GCSs are given in Theorems~\ref{theorem5} and~\ref{theorem6},  which answers the question left in the last part of \cite{Borwein} and part of the first question in \cite{Fiedler2008}. A MATLAB program for the construction of 2-phase and 4-phase GCSs is shared.

\section*{Acknowledgments}
The work was supported in part by The Startup Foundation for Introducing Talent
of NUIST. We acknowledge the High Performance Computing Center of Nanjing University of Information Science and Technology for their support of this work.

\end{document}